\documentclass[11pt,a4paper]{amsart}

\usepackage[T1]{fontenc}
\usepackage[utf8]{inputenc}
\usepackage{lmodern}
\usepackage{mathtools,amssymb}
\usepackage{enumitem}
\usepackage{microtype}
\usepackage[colorlinks=true,linkcolor=blue,citecolor=blue,urlcolor=blue]{hyperref}
\usepackage[nameinlink,noabbrev]{cleveref}
\usepackage[a4paper,margin=1in]{geometry}

\newtheorem{theorem}{Theorem}[section]
\newtheorem{lemma}[theorem]{Lemma}
\theoremstyle{definition}
\newtheorem{definition}[theorem]{Definition}

\newcommand{\eqm}{\equiv_{\mathrm m}}
\newcommand{\lem}{\le_{\mathrm m}}
\newcommand{\lfo}{\le_{\mathrm{fo}}}
\newcommand{\ce}{c.e.}
\newcommand{\om}{\omega}
\newcommand{\comp}[1]{\overline{#1}}

\title[No least finite-one degree in a c.e. many-one degree]{A computably enumerable many-one degree with no least finite-one degree}
\author{Patrizio Cintioli}
\address{Mathematics Division, School of Science and Technology, University of Camerino, Italy}
\email{patrizio.cintioli@unicam.it}


\begin{document}

\begin{abstract}
Richter, Stephan, and Zhang asked whether every nonrecursive many-one degree contains a least finite-one degree. We solve this question in the negative, already within the class of computably enumerable many-one degrees. 
Positive answers are known in two disjoint natural settings: for a measure-one and comeager class of $m$-rigid sets, and, in a companion paper, for computably enumerable many-one degrees containing a $D$-maximal set.
We construct a nonrecursive \ce\ set $A$ such that for every set $X \eqm A$ there exists a c.e.\ set $B \eqm A$ with $X \not\lfo B$. Hence the many-one degree of $A$ contains no least finite-one degree. The proof is a finite-injury priority construction based on virtual target sets and a dynamic trap mechanism forcing any putative finite-one reduction either to violate finite-oneness or to compute an incorrect reduction.
\end{abstract}

\maketitle

\section{Introduction}

Many-one reducibility and its refinements provide a natural framework for studying
the internal structure of many-one degrees. Besides one-one reducibility, the
intermediate notions of finite-one and bounded finite-one reducibility make it
possible to analyse how much finer structure can occur strictly below a many-one
degree while still remaining above its greatest one-one degree. A recent and
broad investigation of these themes is due to Richter, Stephan, and Zhang
\cite{RSZ2026}; see also \cite{Odifreddi1981,OdifreddiCRT,OdifreddiCRT2, Odifreddi1999,Rogers1967, Soare1987}
for general background on strong reducibilities.

Among the questions isolated in \cite{RSZ2026}, the first asks whether every
nonrecursive many-one degree contains a least finite-one degree. This problem is
particularly natural: finite-one reducibility is weak enough to allow genuinely
new internal behaviour inside a many-one degree, yet strong enough to suggest
that some canonical minimum representative might still exist.

There is substantial positive evidence for such a minimum phenomenon, but it
comes from two very different directions. First, in
\cite{Cintioli2026}, it is shown that for a measure-one and comeager class of
$m$-rigid sets, the many-one degree contains a least finite-one degree. Since
$m$-rigid sets are bi-immune, this is a typical phenomenon occurring outside the
computably enumerable world. Second, a companion paper shows that every
nonrecursive c.e.\ many-one degree containing a $D$-maximal set also contains a
least finite-one degree. Thus the known positive evidence covers two disjoint
natural classes: a typical non-c.e.\ class and a structural c.e.\ subclass.

We solve Open Question~1 of \cite{RSZ2026} in the negative, already within
the class of computably enumerable many-one degrees.
More precisely, we prove that there is a nonrecursive c.e.\ many-one degree
in which no member is finite-one reducible to all the others.

\begin{theorem}[Main Theorem]\label{thm:main}
There exists a nonrecursive \ce\ set $A$ such that for every set $X$ many-one
equivalent to $A$, there exists a set $B$ many-one equivalent to $A$ such that
$X$ is not finite-one reducible to $B$. In particular, the many-one degree of
$A$ contains no least finite-one degree.
\end{theorem}

Thus Open Question~1 of \cite{RSZ2026} is settled in the negative.
Indeed, the general existence of a least finite-one degree fails already in the c.e.\ setting.
Combined with the positive results mentioned above, this yields a
more nuanced picture of the landscape suggested by \cite{RSZ2026}: a least
finite-one degree exists throughout important natural subclasses, including a
measure-one and comeager class of sets and a natural c.e.\ structural class,
but it is not a universal feature of nonrecursive many-one degrees; indeed, it
already fails for nonrecursive c.e.\ many-one degrees.

The proof is an asynchronous finite-injury priority construction. 
Given a candidate set $X \eqm A$, coded by a putative reduction $\Delta_j \colon A \lem W_i$, 
we build a corresponding virtual target $B_m \equiv_m A$. 
Against each proposed finite-one reduction from $W_i$ to $B_m$, a dedicated requirement deploys a dynamic trap strategy. 
The construction is arranged so that any such candidate reduction must either fail to be finite-one or fail to compute a correct reduction.

The paper is organised as follows. In
\cref{sec:preliminaries} we fix notation and recall the relevant reducibilities.
In \cref{sec:virtual-targets} we introduce virtual targets and the anchor-block
viewpoint underlying the construction. In \cref{sec:construction} we present the
priority construction, including the totality filler and the dynamic trap
modules. Finally, in \cref{sec:verification} we verify that every valid
candidate $X \eqm A$ is defeated by a suitable target $B_m \eqm A$, thereby
establishing \Cref{thm:main}.

\section{Preliminaries and Notation}\label{sec:preliminaries}

We assume familiarity with standard notions and notation from computability theory; for general background see
\cite{Calude2002,ChongYu2015,DH2010,LiVitanyi2008,Nies2009,OdifreddiCRT,Odifreddi1999,OdifreddiCRT2,Rogers1967,Soare1987}.

Let $\om=\{0,1,2,\dots\}$ denote the set of natural numbers. 
We fix standard effective numberings $(\varphi_e)_{e\in\omega}$ of the partial computable functions and $(W_e)_{e\in\omega}$ of the computably enumerable sets, arranged so that
\[
W_e=\operatorname{dom}(\varphi_e)
\]
for every $e$.
We write $\varphi_e(x)\downarrow$ when the computation halts.
Our construction proceeds in effective stages. We denote by $\varphi_{e,s}(x)$ the stage-$s$ approximation to $\varphi_e(x)$. By convention, if $\varphi_{e,s}(x)\downarrow$, then the computation halts in strictly fewer than $s$ steps, we have $x<s$, and its output is strictly less than $s$.
We also fix effective numberings $(\Delta_j)_{j\in\omega}$ and $(\Phi_k)_{k\in\omega}$
of all partial computable functions from $\omega$ to $\omega$. Thus every total
computable function appears as some $\Delta_j$ and as some $\Phi_k$.
As with $\varphi_e$, we denote their finite approximations evaluated by stage $s$ as $\Delta_{j,s}$ and $\Phi_{k,s}$.
The numbering $(\Delta_j)_{j\in\omega}$ will be used for candidate many-one reductions, typically witnessing reductions of the form
\[
A \lem W_i,
\]
whereas $(\Phi_k)_{k\in\omega}$ will be used for candidate finite-one reductions from $W_i$ to the virtual targets constructed later.
For a set $A \subseteq \om$, its complement is denoted by $\comp{A} = \om \setminus A$.
We implicitly identify a set with its characteristic function, so $A\neq \varphi_n$ means that $\varphi_n$ is not the characteristic function of $A$.

\subsection{Reducibilities and Degrees}
For sets $A, B \subseteq \om$, we say that $A$ is \emph{many-one reducible} to $B$, denoted $A \lem B$, if there exists a total computable function $f \colon \om \to \om$ such that $x \in A \iff f(x) \in B$ for all $x \in \om$. If both $A \lem B$ and $B \lem A$, we say that $A$ and $B$ are many-one equivalent, denoted $A \eqm B$.
The many-one degree (or m-degree) of a set $A$ is its equivalence class under $\eqm$.

By imposing finiteness constraints on the fibres of the reduction function $f$, we obtain finer reducibilities.

\begin{definition}
A total computable function $f \colon \om \to \om$ is called \emph{finite-one} if the preimage of every element is finite; that is, for all $y \in \om$, the set $f^{-1}(y) = \{x \in \om \mid f(x) = y\}$ is finite. 
\end{definition}

We write $A \lfo B$ if there is a finite-one reduction from $A$ to $B$.

\begin{definition}
The equivalence classes induced by mutual finite-one reducibility are the finite-one degrees.
A many-one degree $\mathbf{a}$ is said to contain a \emph{least finite-one degree} if there exists $X \in \mathbf{a}$ such that for every $Y \in \mathbf{a}$, we have $X \lfo Y$.
Equivalently, the finite-one degree of $X$ is least among those represented inside $\mathbf{a}$.
\end{definition}

The existence of such an $X$ means that $X$ is minimal for finite-one reducibility within its own m-degree. Accordingly, to prove that the m-degree of a set $A$ contains no least finite-one degree, it suffices to show that for every $X \eqm A$ there exists $B \eqm A$ such that $X \not\lfo B$.

\subsection{Columns and Fresh Elements}
We fix a standard bijective computable pairing function $\langle \cdot, \cdot \rangle \colon \om \times \om \to \om$.
Using a standard computable coding of finite tuples, we fix pairwise disjoint infinite computable sets
\[
\omega^{[m,k]}=\{\langle 0,m,k,n\rangle : n\in\omega\}
\quad\text{and}\quad
\omega^{[*]}=\{\langle 1,n\rangle : n\in\omega\}.
\]
We call these sets \emph{columns}.
We refer to $\omega^{[*]}$ as the \emph{neutral column}, which will be used exclusively to supply witnesses for the requirements $\mathcal D_n$ introduced below.

Throughout the priority argument, when we specify to pick a \emph{``fresh''} element, we mean choosing a natural number strictly larger than any number mentioned, evaluated, or used in the construction up to the current stage $s$, including the stage number $s$ itself. This strict monotonicity guarantees that fresh elements avoid any previously established constraints or pre-existing block links.

\section{Virtual Targets and Current Anchor Blocks}\label{sec:virtual-targets}

This section introduces the virtual targets $B_m$ and the anchor-block viewpoint
that underlies the construction. The purpose is to isolate the mechanism that
will later ensure $B_m \equiv_m A$ for every valid parameter $m$, while keeping
the action on the target indirect, through the coding maps $\Theta_m$ and $\Lambda_m$.
Suppose an opponent claims $X \equiv_m A$. Then $X=W_i$ for some index $i$, and
there exists a total many-one reduction $\Delta_j \colon A \lem W_i$. We code
this candidate by the parameter $m=\langle i,j\rangle$.

For each $m$ we build partial computable approximations $\Theta_{m,s}$ and
$\Lambda_{m,s}$. Whenever $\Delta_j$ is total, these approximations are 
monotonically extended to total computable functions $\Theta_m$ and $\Lambda_m$. We then define
\[
B_m=\{y\in\omega : \Lambda_m(y)\in A\}.
\]
By definition, $\Lambda_m$ reduces $B_m$ to $A$.

To obtain $\Theta_m \colon A \lem B_m$, it is enough to maintain
\[
(\forall x)\qquad x\in A \iff \Lambda_m(\Theta_m(x))\in A.
\]
Write
\[
\sigma_m=\Lambda_m\circ\Theta_m.
\]
The construction is arranged so that, for every valid parameter $m=\langle i,j\rangle$, the final map $\sigma_m=\Lambda_m\circ\Theta_m$ satisfies the following dichotomy:

\begin{quote}
For every $x$, exactly one of the following holds: either $x$ is never tied in
any run (the \emph{private case}, yielding $\sigma_m(x)=x$), or $x$ is tied in a
unique run $\rho$ of a unique requirement $l=\langle m,k\rangle$, in which case
$\sigma_m(x)=a_1^{\, l,\rho}$, where $a_1^{\, l,\rho}$ is the anchor chosen in that
run.
Moreover, whenever tied elements from a given run are enumerated into
$A$, the entire finite current anchor block of that run is enumerated
simultaneously.
\end{quote}

Thus, for the purpose of the reduction $A\lem B_m$, the only nontrivial identifications
created by $\sigma_m$ are among witnesses tied in a single run of a single requirement.
The filler routine eventually makes all other points private.
In particular, the construction never uses an ambient equivalence relation on
$\omega$; only the current finite anchor blocks created during individual runs
matter.

The set $B_m$ is therefore a virtual target: we do not enumerate it directly, but only act on $A$ and on the coding maps $\Theta_m$ and $\Lambda_m$. To defeat a candidate finite-one reduction $\Phi_k\colon W_i\to B_m$, the requirement $\mathcal R_{\langle m,k\rangle}$ chooses a trap point $b^*$ and successive fresh baits $a_v$. If a collision $x_v=x_u$ occurs, the run terminates immediately and this already shows that $\Delta_j$ was not a correct reduction. Otherwise, whenever $x_v=\Delta_j(a_v)$ is new, it ties $a_v$ to $b^*$, so that $\sigma_m(a_v)=a_1$.
If some computation $\Phi_k(x_v)$ never converges, then $\Phi_k$ is not total and hence cannot be a finite-one reduction.
If $\Phi_k(x_v)=b^*$ for infinitely many distinct values $x_v$, then $\Phi_k$ is not finite-one. Otherwise, on the first escape value $y_v\neq b^*$, the construction waits until $y_v$ is not frozen and $\Lambda_m(y_v)$ is defined, lets $c=\Lambda_m(y_v)$, and then acts according to whether $c$ has entered $A$ by that stage, thereby forcing $\Phi_k$ to disagree with $W_i\to B_m$.

\section{The Priority Construction}\label{sec:construction}

We now present the finite-injury priority construction witnessing the main theorem.
After specifying the requirements, we describe the global stage protocol, the
totality filler, and the two kinds of modules implementing diagonalization and
the dynamic trap strategy.
To ensure that $A$ is nonrecursive, we satisfy the standard diagonalization
requirements
\begin{description}
    \item[$\mathcal{D}_n$] If $\varphi_n$ is a total computable function, then
    $A\neq \varphi_n$.
\end{description}

To defeat finite-one candidates, we satisfy the negative requirements
\begin{description}
    \item[$\mathcal{R}_l$] Let $l=\langle m,k\rangle$ and
    $m=\langle i,j\rangle$. If $\Delta_j\colon A\lem W_i$ is a total
    many-one reduction, then $\Phi_k$ is not a finite-one reduction from
    $W_i$ to $B_m$.
\end{description}

We interleave the requirements into a single priority list
\[
\mathcal Q_0,\mathcal Q_1,\mathcal Q_2,\dots
\]
where $\mathcal Q_{2n}=\mathcal D_n$ and $\mathcal Q_{2l+1}=\mathcal R_l$.

\paragraph{\textbf{Conventions for the modules.}}
Each requirement carries a persistent \emph{local state} together with a finite
tuple of \emph{local parameters}. These data are part of the construction and
persist from stage to stage until the requirement is initialized by a
higher-priority requirement.
At the start of the construction, every requirement is in State~1 with empty local data.

A \emph{run} of a requirement begins either at the start of the construction or when the
requirement is initialized, and ends either when the requirement is next initialized or when
it performs the terminal action of that run, whichever comes first.

To \emph{initialize} a requirement means to
discard any stored local parameters from its previous run, unfreeze any bait currently
frozen by that run, and reset its local state to State~1.
If a requirement has not declared itself satisfied and is not initialized, then
at the next stage it resumes from exactly the same state with exactly the same stored parameters.

When a requirement of priority index $p$ \emph{initializes all lower-priority
requirements} during a stage $s$, this means that every requirement
$\mathcal Q_q$ with $p<q\le s$ is initialized immediately. No action is needed
for $q>s$, since such requirements have not yet been visited and are still in
State~1 with empty local data.

Accordingly, instructions such as ``pick a fresh witness'', ``pick a fresh
bait'', or ``pick a trap point'' are understood to mean: pick such an object
only if it has not already been chosen in the current run, and then keep it
fixed for the rest of that run unless the requirement is initialized.
Similarly, a clause of the form ``wait until \dots'' means that the requirement
remains in its current state and takes no further action until the stated
condition is met.

A requirement \emph{declares itself satisfied} only when it performs the
terminal action of its current run. Once satisfied, it remains inactive unless
it is later initialized by a higher-priority requirement.

For $\mathcal D_n$, the local data of a run consist only of its current witness
$x_n$.
For $\mathcal R_l$, the local data of a run consist of its current
counter $v$, its trap point $b^*$, the finite list of baits already chosen
in the current run, the corresponding values $x_r$ for those chosen baits
whose computations $\Delta_j(a_r)$ have converged, and the current value
$y_v=\Phi_k(x_v)$ whenever it has converged.
When $\mathcal R_l$ returns from State~4 to
State~2 in the branch $y_v=b^*$, this does \emph{not} start a new run: the
requirement keeps the same trap point $b^*$, retains all previously chosen
baits together with their corresponding converged values, increments $v$,
forgets only the temporary value $y_v$, and continues the same run by choosing
a new bait for the new value of $v$.

\paragraph{\textbf{Operational semantics of a stage.}}
At stage $s$ we visit the modules
\[
\mathcal Q_0,\mathcal Q_1,\dots,\mathcal Q_s
\]
in priority order, and only after that do we execute the filler for all
parameters $m,z\le s$.
When a module is visited at stage $s$, it performs one \emph{local routine},
possibly empty if the module is currently inactive.
A local routine is a finite block of
primitive updates carried out atomically before the next module is visited and
before the filler is run. Primitive updates may:
\begin{enumerate}[label=\textup{(\arabic*)}, leftmargin=2.5em]
\item change the local state and local parameters of the module;
\item freeze or unfreeze finitely many currently stored baits;
\item enumerate finitely many numbers into $A$;
\item define finitely many previously undefined values of $\Theta_m$ or
$\Lambda_m$;
\item initialize finitely many lower-priority requirements among
$\mathcal Q_{p+1},\dots,\mathcal Q_s$, where $p$ is the priority index of the
current module.
\end{enumerate}
Here and below, the phrase ``initializes all lower-priority requirements'' is understood in the sense fixed above.

No value of $\Theta_m$ or $\Lambda_m$ is ever changed once defined.

If, during its local routine, a module changes from one state to the next, then
the instructions of the new state are executed immediately in the same stage,
in the written order, until the routine reaches an explicit waiting clause or a
terminal action. Thus phrases such as ``move immediately to State~2'',
``once this happens, move to State~3'', and ``upon entering State~3'' are to be
understood literally within one stage.

In particular, in the noncollision branch of State~3 for $\mathcal R_l$, the
assignment
\[
\Theta_m(a_v)=b^*
\]
is performed before $a_v$ is unfrozen. Hence the filler can never act in
between these two steps and cannot privatize $a_v$. Likewise, in the collision
branch, the unfreezing of $a_v$, the enumeration of the already tied anchor
block, and the initialization of lower-priority requirements all occur before
the filler of that stage.

The passive loop in State~4 is still stage-by-stage rather than instantaneous.
Indeed, if at stage $s$ the branch $y_v=b^*$ returns the module from State~4 to
State~2, then the newly chosen bait for the incremented value of $v$ is fresh and
therefore strictly larger than $s$.
By the stage convention on computations, the computation
\[
\Delta_{j,s}(a_v)
\]
cannot already have converged for this new bait,
so the local routine necessarily stops by waiting in State~2.
Thus every local routine is finite.

\subsection{The Totality Filler}

At the end of stage $s$, process all pairs $(m,z)$ with $m,z\le s$
in some fixed computable order, say lexicographic order, and for each such pair do the following.

\begin{enumerate}
\item If $z\notin \bigcup_k \omega^{[m,k]}$ and neither $\Theta_m(z)$ nor $\Lambda_m(z)$
is defined, put
\[
\Theta_m(z)=z,\qquad \Lambda_m(z)=z.
\]

\item If $z\in\omega^{[m,k]}$ for some $k$ and $z$ is not currently frozen for $m$ by any requirement, then perform the following checks in order.
\begin{enumerate}
\item If $\Theta_m(z)$ is undefined, pick a fresh $y\in\omega^{[m,k]}$ and set
\[
\Theta_m(z)=y,\qquad \Lambda_m(y)=z.
\]

\item If $\Lambda_m(z)$ is undefined, pick a fresh $x\in\omega^{[m,k]}$ and set
\[
\Lambda_m(z)=x,\qquad \Theta_m(x)=z.
\]
\end{enumerate}
\end{enumerate}

All fresh numbers are chosen larger than the current stage and larger than every
number mentioned so far. No value of $\Theta_m$ or $\Lambda_m$ is ever changed
after being assigned.

\subsection{Module for \texorpdfstring{$\mathcal{D}_n$}{Dn}}

When $\mathcal D_n$ is initialized, it discards any stored witness (if any) and
enters State~1.
State~2 performs the terminal action of the current run. After this action,
$\mathcal D_n$ remains inactive unless it is later initialized by a
higher-priority requirement.

\begin{description}[leftmargin=*, style=nextline]
    \item[State 1 (waiting)]
    If no current witness has yet been chosen in the current run, pick a fresh
    witness $x_n\in\omega^{[*]}$. Wait until $\varphi_{n,s}(x_n)\downarrow$.
    Once this happens, move to State~2.

    \item[State 2 (terminal action)]
    Upon entering State~2:
    \begin{itemize}
        \item If $\varphi_{n,s}(x_n)=0$, enumerate $x_n$ into $A$.
        \item If $\varphi_{n,s}(x_n)\neq 0$, do nothing.
    \end{itemize}
    Declare satisfied and initialize all lower-priority requirements.
\end{description}

\subsection{Module for \texorpdfstring{$\mathcal{R}_l$}{Rl} (Dynamic Trap)}

Let $l=\langle m,k\rangle$ and $m=\langle i,j\rangle$.
The requirement $\mathcal R_l$ chooses all its baits inside the column
$\omega^{[m,k]}$, and any elements that it enumerates into $A$ from this run
also lie in that column.
When $\mathcal R_l$ is initialized, it discards all stored data from its most recent run
(in particular, when present, its current counter $v$, its trap point $b^*$,
the finite list of baits already chosen in that run, the corresponding values
$x_r$ for those chosen baits whose computations $\Delta_j(a_r)$ have converged,
and the current value $y_v$ if present), unfreezes any bait currently frozen by
that run, and enters State~1.
If the current run has not yet performed its terminal action, these data persist
from stage to stage unless $\mathcal R_l$ is initialized. After the terminal
action of a run, $\mathcal R_l$ becomes inactive, and its stored data are
discarded only when it is next initialized.

\begin{description}[leftmargin=*, style=nextline]
\item[State 1 (start of a run)]
Set $v=1$ and move immediately to State~2.

\item[State 2 (choose the current bait and wait for $\Delta_j$)]
If the current bait $a_v$ has not yet been chosen in this run, pick a fresh
$a_v\in\omega^{[m,k]}$ and freeze it for $m$.
If $v=1$ and the trap point $b^*$ has not yet been chosen in this run, pick a
fresh $b^*\in\omega^{[m,k]}$ and define
\[
\Lambda_m(b^*)=a_1.
\]
Check whether $\Delta_{j,s}(a_v)\downarrow=x_v$. Once this happens, move to State~3.

\item[State 3 (collision test, tie step, and wait for $\Phi_k$)]

If $x_v=x_u$ for some $u<v$, then unfreeze $a_v$ and leave it untied (so the filler will later
make it private), enumerate the already tied anchor block
\[
C^-_{l,s}=\{a_r : r<v \text{ and } \Theta_{m,s}(a_r)=b^*\}
\]
into $A$, declare satisfied, and initialize all lower-priority requirements.

If $x_v$ is distinct from all previous $x_u$, define
\[
\Theta_m(a_v)=b^*.
\]
Now unfreeze $a_v$. Thus $a_v$ is safely tied to the anchor $a_1$. Now check whether
\[
\Phi_{k,s}(x_v)\downarrow=y_v.
\]
Once this happens, move to State~4.

\item[State 4 (trap evaluation)]
If $y_v=b^*$, then keep the same trap point $b^*$ and all previously chosen
baits $a_1,\dots,a_v$ together with their values $x_1,\dots,x_v$, forget only
the temporary value $y_v$, set $v\gets v+1$, and return to State~2, where a
new bait for the new value of $v$ will be chosen. This is a passive loop and
does \emph{not} initialize lower-priority requirements.

If $y_v\neq b^*$, wait until $y_v$ is not frozen
for $m$ by any requirement and $\Lambda_{m,s}(y_v)$ is defined.
(If $\Lambda_m(y_v)$ is undefined and $y_v$ later becomes unfrozen for $m$,
then a later execution of the filler will define it.)

Once this happens, let
\[
c=\Lambda_{m,s}(y_v).
\]

If $c\notin A_s$, enumerate the current tied anchor block
\[
C_{l,s}=\{a_r : r\le v \text{ and } \Theta_{m,s}(a_r)=b^*\}
\]
into $A$.
If $c\in A_s$, do nothing.

Declare satisfied and initialize all lower-priority requirements.
\end{description}

\section{Verification}\label{sec:verification}

We now verify that the construction works as intended. First we show that every
valid parameter $m$ yields a c.e.\ set $B_m \equiv_m A$; next we prove finite
injury and the noncomputability of $A$; finally we verify that every candidate
finite-one reduction from $W_i$ to $B_m$ is defeated.

\begin{lemma}\label{lem:Bm-equivalent}
Fix $m=\langle i,j\rangle$ such that $\Delta_j\colon A\lem W_i$ is a total
computable reduction. Then the construction defines total computable functions
$\Theta_m$ and $\Lambda_m$, and the c.e. set
\[
B_m=\Lambda_m^{-1}(A)
\]
satisfies $B_m\equiv_m A$.
\end{lemma}

\begin{proof}
We first show that every number is eventually permanently unfrozen for $m$.

An element is frozen for $m$ only while it is serving as the current bait of
some requirement $\mathcal R_{\langle m,k\rangle}$ in State~2. Since
$\Delta_j$ is total, such a bait cannot remain frozen forever: either that run
eventually reaches State~3, in which case the bait is unfrozen there, or the
run is initialized earlier by a higher-priority requirement, in which case the
bait is unfrozen at initialization. By freshness, once a number has been chosen
as a bait, it is never chosen again. Hence every number is frozen for $m$ only
for finitely many stages, and from some stage onward it is permanently unfrozen.

Now fix $z$. Once the stage is large enough that $z\le s$ and $z$ is no longer
frozen for $m$, the filler acts on $z$ at every later stage until all still
undefined values among $\Theta_m(z)$ and $\Lambda_m(z)$ have been supplied.
(Some of these values may already have been defined earlier by a requirement
module, for instance $\Theta_m(a_v)=b^*$ or $\Lambda_m(b^*)=a_1$.)
Therefore both $\Theta_m(z)$ and $\Lambda_m(z)$ are eventually defined for
every $z$.

Since values are assigned at most once, both graphs are c.e.\ and single-valued.
Because they are total, it follows that $\Theta_m$ and $\Lambda_m$ are total
computable functions.
The set
\[
B_m=\Lambda_m^{-1}(A)
\]
is c.e.\ because $A$ is c.e.\ and $\Lambda_m$ is total computable. By
definition, $\Lambda_m$ reduces $B_m$ to $A$.

It remains to show that $\Theta_m$ reduces $A$ to $B_m$. Fix $x\in\omega$ and
write
\[
\sigma_m=\Lambda_m\circ\Theta_m.
\]

We claim that exactly one of the following two alternatives holds.

\smallskip
\noindent\emph{Private case.}
Assume that $x$ is never tied in any run of any requirement
$\mathcal R_{\langle m,k\rangle}$. Then
\[
\sigma_m(x)=x.
\]

Indeed, since $x$ is never tied, the only possible way to define $\Theta_m(x)$
is through the filler. When the filler defines $\Theta_m(x)$, one of two things
happens.

Either the filler acts directly on $z=x$ (via clause 1 or 2a) and writes
\[
\Theta_m(x)=y,
\qquad
\Lambda_m(y)=x.
\]

Or else $\Theta_m(x)$ is created as the companion assignment (via clause 2b) to some definition
\[
\Lambda_m(z)=x,
\qquad
\Theta_m(x)=z.
\]

In either case,
\[
\Lambda_m(\Theta_m(x))=x.
\]
Since no module ever ties $x$, no later action can replace this value of
$\Theta_m(x)$. Hence $\sigma_m(x)=x$.

\smallskip
\noindent\emph{Tied case.}
Assume that $x$ is tied at some stage. By freshness, this can happen only once:
there is a unique requirement $l=\langle m,k\rangle$ and a unique run $\rho$ of
that requirement such that $x$ is one of the baits chosen in that run and is
tied there. Let
\[
a_1^{\,l,\rho}
\qquad\text{and}\qquad
b^{*,l,\rho}
\]
be the anchor and trap point of that run.

Since $x$ was fresh when chosen as a bait in the run $\rho$, the value
$\Theta_m(x)$ was still undefined at that moment. By the operational semantics
of a stage, in the noncollision branch of State~3 the assignment
\[
\Theta_m(x)=b^{*,l,\rho}
\]
is made before $x$ is unfrozen. Hence the filler never gives $x$ a private
$\Theta$-value.
Also, when the trap point of that run is chosen, the
construction defines
\[
\Lambda_m(b^{*,l,\rho})=a_1^{\,l,\rho}.
\]
Therefore
\[
\sigma_m(x)=a_1^{\,l,\rho}.
\]
Thus, to prove that $\Theta_m$ reduces $A$ to $B_m$, it is enough to show that
whenever $x$ is tied in a run $\rho$ as above,
\[
x\in A \iff a_1^{\,l,\rho}\in A. \tag{$*$}
\]

Fix such a tied $x$. We now analyze the fate of the run $\rho$.

Before doing so, note that because $m$ is valid, whenever a run of some
$\mathcal R_{\langle m,k\rangle}$ reaches State~4 with a value $y_v\neq b^*$
and is not initialized first, the waiting clause in State~4 is eventually met:
every number is eventually permanently unfrozen for $m$, and once $y_v$ is
unfrozen, the filler eventually defines $\Lambda_m(y_v)$ if it is still
undefined.

There are now five possibilities.

\begin{enumerate}
\item \emph{The run $\rho$ is initialized by a higher-priority requirement
before it performs any terminal action.}

Then neither $x$ nor $a_1^{\,l,\rho}$ ever enters $A$. Indeed, both were fresh
when chosen in the run $\rho$, so neither had previously been enumerated into
$A$. Moreover, only $\mathcal R_l$ can enumerate elements from the dedicated
column $\omega^{[m,k]}$; initialization abandons the old local data of $\rho$;
and every later bait of $\mathcal R_l$ is fresh and therefore different from
both $x$ and $a_1^{\,l,\rho}$.

\item \emph{The run $\rho$ terminates by a collision in State~3.}

Since $x$ is tied, it cannot be the current collision witness $a_v$, because
the collision branch explicitly leaves $a_v$ untied. So $x=a_u$ for some
$u<v$. Let $s$ be the stage at which this collision occurs.
At stage $s$, both $x$ and the anchor
$a_1^{\,l,\rho}$ belong to the already tied anchor block
\[
C^-_{l,s}=\{a_r:r<v\text{ and }\Theta_{m,s}(a_r)=b^{*,l,\rho}\},
\]
and that block is enumerated into $A$. Hence $x$ and $a_1^{\,l,\rho}$ enter
$A$ together.

\item \emph{The run $\rho$ reaches State~4 with some $y_v\neq b^{*,l,\rho}$
and terminates there in the branch $c\notin A_t$.}

Let $t$ be the stage at which this terminal action occurs. At stage $t$, the current tied anchor block
\[
C_{l,t}=\{a_r:r\le v\text{ and }\Theta_{m,t}(a_r)=b^{*,l,\rho}\}
\]
is enumerated into $A$. Since $x$ is one of the tied baits of this run, both
$x$ and $a_1^{\,l,\rho}$ belong to $C_{l,t}$. Hence again they enter $A$
together.

\item \emph{The run $\rho$ reaches State~4 with some $y_v\neq b^{*,l,\rho}$
and terminates there in the branch $c\in A_t$.}

Let $t$ be the stage at which this terminal action occurs. 

Then the run performs no enumeration of its current anchor block. Since the
elements of that block were fresh when chosen, none of them had been previously
enumerated into $A$. Later initializations do not change this conclusion: only
$\mathcal R_l$ ever enumerates elements from the column $\omega^{[m,k]}$, and
every later bait of $\mathcal R_l$ is fresh. Hence no element of the old
anchor block of $\rho$ is ever enumerated into $A$. In particular, both
$x$ and $a_1^{\,l,\rho}$ remain outside $A$.
\item \emph{After $x$ has been tied, the run $\rho$ never performs a terminal
action.}

This can happen in two ways (note that the module cannot wait forever in State~2
because $\Delta_j$ is total, nor in State~4 by the observation above):
either the module waits forever in State~3 because
some computation $\Phi_k(x_v)$ never converges, or it returns from State~4 to
State~2 infinitely often through the passive loop $y_v=b^*$.
If the run is later initialized, we are already in case~(1). Otherwise no
terminal enumeration ever occurs, and only $\mathcal R_l$ could enumerate
elements from the column $\omega^{[m,k]}$. Since $x$ and $a_1^{\,l,\rho}$ were
fresh when chosen, neither had previously entered $A$. Therefore both
$x$ and $a_1^{\,l,\rho}$ remain outside $A$.

\end{enumerate}

In every case, $(*)$ holds. Combining the private case and the tied case, we
obtain for every $x\in\omega$:
\[
x\in A \iff \sigma_m(x)\in A.
\]
Since
\[
\sigma_m(x)=\Lambda_m(\Theta_m(x)),
\]
this is equivalent to
\[
x\in A \iff \Theta_m(x)\in B_m.
\]
Therefore $\Theta_m$ reduces $A$ to $B_m$. Together with the reduction
$\Lambda_m\colon B_m\lem A$, this yields
\[
A\equiv_m B_m.
\]
\end{proof}

\begin{lemma}\label{lem:finite-injury}
Every requirement is initialized only finitely many times. Consequently each
diagonalization requirement $\mathcal D_n$ is met, and $A$ is nonrecursive.
\end{lemma}

\begin{proof}
We argue by induction on the priority ordering.

Fix $p$, and assume that every requirement $\mathcal Q_q$ with $q<p$
initializes lower-priority requirements only finitely many times. Let $s_0$ be
a stage after the last such action. After stage $s_0$, the requirement
$\mathcal Q_p$ is never initialized again.

If $\mathcal Q_p=\mathcal D_n$, then we distinguish two cases.

If $\mathcal D_n$ has already declared satisfied by stage $s_0$, then it remains
inactive forever and never initializes lower-priority requirements again.

Otherwise, from stage $s_0$ onward it works within its final run and hence with
a final witness $x_n$ (retaining the one already stored, or choosing one at its
next visit if none has yet been chosen). If the computation $\varphi_n(x_n)$
eventually converges, then $\mathcal D_n$ performs the terminal action of this
final run and initializes lower-priority requirements once. If
$\varphi_n(x_n)$ never converges, it waits forever in State~1 and never
initializes lower-priority requirements again.

Hence after stage $s_0$, the requirement $\mathcal D_n$ initializes
lower-priority requirements at most once. Since there are only finitely many
stages before $s_0$, it follows that $\mathcal D_n$ initializes lower-priority
requirements only finitely many times in total.

If $\mathcal Q_p=\mathcal R_l$, then we again distinguish two cases.

If $\mathcal R_l$ has already declared satisfied by stage $s_0$, then it remains
inactive forever and never initializes lower-priority requirements again.

Otherwise, from stage $s_0$ onward it works within its final run. Since State~1
moves immediately to State~2, there are five possibilities. It may wait forever
in State~2 for some computation $\Delta_j(a_v)$ to converge; or it may wait
forever in State~3 for some computation $\Phi_k(x_v)$ to converge; or it may
wait forever in State~4 for some value $y_v$ to become unfrozen for $m$ or for
$\Lambda_m(y_v)$ to be defined; or it may run forever in the passive trap loop;
or it may eventually perform the terminal action of this final run (by
collision or trap exit). In the first four cases it never initializes
lower-priority requirements again. In the fifth case it does so exactly once.

Hence after stage $s_0$, the requirement $\mathcal R_l$ initializes
lower-priority requirements at most once. Since there are only finitely many
stages before $s_0$, it follows that $\mathcal R_l$ initializes lower-priority
requirements only finitely many times in total.
This completes the induction.

Therefore every requirement initializes lower-priority requirements only finitely
many times. Since a requirement $\mathcal Q_p$ can be initialized only by one of
the finitely many higher-priority requirements $\mathcal Q_q$ with $q<p$, it
follows that every requirement is initialized only finitely many times.

Now fix $n$. After the last initialization of $\mathcal D_n$, its final witness
$x_n$ is never abandoned again. Assume that $\varphi_n$ is total.

If $\varphi_n(x_n)=0$, then $\mathcal D_n$ enumerates $x_n$ into $A$, so
$A(x_n)=1\neq \varphi_n(x_n)$.

If $\varphi_n(x_n)\neq 0$, then $\mathcal D_n$ performs no enumeration at its
terminal stage. We claim that $x_n$ remains permanently outside $A$.

First, $x_n$ was fresh when chosen, so $x_n\notin A$ at that moment.
Before the terminal stage of $\mathcal D_n$, no other requirement can enumerate
$x_n$ into $A$: every negative requirement $\mathcal R_l$ acts only inside its
dedicated column $\omega^{[m,k]}$, which is disjoint from the neutral column
$\omega^{[*]}$, while every positive requirement $\mathcal D_r$ can enumerate
only its own witness, and every other such witness is different from $x_n$
(earlier ones by the freshness of $x_n$ when chosen, later ones by their own
freshness).

Now consider stages after the terminal action of $\mathcal D_n$. No
higher-priority requirement later performs a terminal action, for any such
action would initialize $\mathcal D_n$ again. At the terminal stage of
$\mathcal D_n$, every lower-priority requirement that has already been visited
is initialized, hence abandons its old local data; and any lower-priority
requirement first visited only later chooses its witnesses fresh, again
different from $x_n$. As above, negative requirements never act inside
$\omega^{[*]}$.

Hence no later action can enumerate $x_n$ into $A$. Therefore
$A(x_n)=0\neq \varphi_n(x_n)$.

Thus $\mathcal D_n$ is met. Since $n$ was arbitrary, $A$ differs from every
total computable function. Therefore $A$ is nonrecursive.
\end{proof}

\begin{lemma}\label{lem:no-fin-one}
Fix $m=\langle i,j\rangle$ such that $\Delta_j\colon A\lem W_i$ is a total
computable reduction. Then
for every $k$, $\Phi_k$ is not a finite-one reduction from $W_i$ to $B_m$.
\end{lemma}

\begin{proof}
Let $l=\langle m,k\rangle$.
By Lemma~\ref{lem:finite-injury}, the requirement $\mathcal R_l$ is initialized
only finitely many times by higher-priority requirements. Let $s_0$ be the
stage of the last such initialization, if there is one, and put $s_0=2l+1$
otherwise. Let $\rho$ be the run of $\mathcal R_l$ active from stage $s_0$
onward. This is the final run.

No higher-priority requirement performs a terminal action at any later stage.
Indeed, if some higher-priority requirement of index $p<2l+1$ acted terminally
at a stage $t>s_0$, then it would initialize every requirement $\mathcal Q_q$
with $p<q\le t$; since $t\ge 2l+1$, this would include $\mathcal R_l$,
contradicting the choice of $s_0$.

We show that $\Phi_k$ cannot be a finite-one reduction from $W_i$ to $B_m$.

\smallskip
\noindent\emph{Step 1: no collision occurs in the final run.}

Suppose toward a contradiction that at some stage $s$ of the final run there is
a collision
\[
x_v=x_u
\qquad\text{with }u<v.
\]
Then $a_u$ has already passed through the noncollision branch of State~3 and is
therefore tied, whereas $a_v$ is still untied. At stage $s$ the collision
branch enumerates the already tied anchor block
\[
C^-_{l,s}=\{a_r:r<v\text{ and }\Theta_{m,s}(a_r)=b^*\}
\]
into $A$, so in particular
\[
a_u\in A.
\]

On the other hand, $a_v$ was fresh when chosen, so $a_v\notin A$ at that
moment. Before stage $s$, only $\mathcal R_l$ could enumerate elements from the
column $\omega^{[m,k]}$, and while serving as the current bait in State~2,
$a_v$ is frozen and is not enumerated into $A$. At stage $s$, the collision
branch explicitly unfreezes $a_v$ and leaves it untied.
Since this is the final run, no higher-priority terminal action will occur
later. At stage $s$, every lower-priority requirement $\mathcal Q_q$ with
$q>2l+1$ and $q\le s$ is initialized immediately and hence abandons its old
local data; thereafter, whether later in the same stage or at later stages, any
witness or bait newly chosen by a lower-priority requirement is fresh and
therefore different from $a_v$. Moreover, only $\mathcal R_l$ can enumerate
elements from the column $\omega^{[m,k]}$.
Hence
\[
a_v\notin A
\]
permanently.
But
\[
\Delta_j(a_u)=x_u=x_v=\Delta_j(a_v),
\]
contradicting that $\Delta_j$ reduces $A$ to $W_i$. Therefore no collision
occurs in the final run, and the values $x_v$ are pairwise distinct.

\smallskip
\noindent\emph{Step 2: all relevant values $\Phi_k(x_v)$ converge, or else
$\Phi_k$ is not a finite-one reduction.}

If for some $v$ reached in State~3 of the final run the computation
$\Phi_k(x_v)$ diverges, then the module waits forever in State~3. In
particular, $\Phi_k$ is not total, and hence cannot be a finite-one reduction
from $W_i$ to $B_m$. Therefore, for the rest of the proof, assume that every
value $\Phi_k(x_v)$ arising in the final run converges.

\smallskip
\noindent\emph{Step 3: the passive loop already destroys finite-oneness if it
occurs infinitely often.}

If State~4 returns to State~2 infinitely often, then
\[
\Phi_k(x_v)=b^*
\]
for infinitely many pairwise distinct inputs $x_v$ (pairwise distinct by
Step~1). Hence the fibre
\[
\Phi_k^{-1}(\{b^*\})
\]
is infinite, so $\Phi_k$ is not finite-one.

Therefore we may assume that State~4 returns to State~2 only finitely often.
Since $m$ is valid, the final run cannot wait forever in State~2; by Step~2 it
cannot wait forever in State~3; and by Step~1 no collision occurs. Hence there
is a first $v$ in the final run such that
\[
\Phi_k(x_v)\downarrow=y_v\neq b^*.
\]

\smallskip
\noindent\emph{Step 4: choose the trap-exit stage.}

Because $m$ is valid, every number is eventually permanently unfrozen for $m$.
Hence $y_v$ is eventually permanently unfrozen for $m$. Moreover, once $y_v$ is
unfrozen, the filler eventually defines $\Lambda_m(y_v)$ if it is still
undefined. Since this is the final run, $\mathcal R_l$ is never initialized
again. Therefore there is a first stage $t$ at which the waiting clause in
State~4 is met, namely such that $y_v$ is not frozen for $m$ and
$\Lambda_{m,t}(y_v)$ is defined. Let
\[
c=\Lambda_{m,t}(y_v).
\]

\smallskip
\noindent\emph{Step 5: $c$ lies outside the current anchor block.}

Let
\[
C_{l,t}=\{a_r:r\le v\text{ and }\Theta_{m,t}(a_r)=b^*\}
\]
be the current anchor block of the final run at stage $t$. We claim that
\[
c\notin C_{l,t}.
\]

First, $c\neq a_1$. The assignment
\[
\Lambda_m(b^*)=a_1
\]
is made when the trap point of the current run is chosen. We claim that $b^*$
is the unique point mapped to $a_1$ by $\Lambda_m$.

No other module instruction can assign the value $a_1$ to $\Lambda_m$. State~3
defines only $\Theta$-values, and State~4 defines no new values of $\Theta_m$
or $\Lambda_m$. If some run of a requirement $\mathcal R_{\langle m,k'\rangle}$
defines a value of $\Lambda_m$ in State~2, that value is its own anchor
$a'_1$. For $k'\neq k$, we have $a'_1\in\omega^{[m,k']}$ and hence
$a'_1\neq a_1$; for other runs of $\mathcal R_l$, freshness again ensures that
their anchors are different from $a_1$.

As for the filler, clause \textup{(1)} never applies to $a_1$, since
$a_1\in\omega^{[m,k]}$. A new preimage of $a_1$ under $\Lambda_m$ could be
created by clause \textup{(2a)} only by acting on $z=a_1$. No such earlier
action occurred, since $a_1$ was later chosen fresh above every number
previously mentioned. After $a_1$ is chosen, it is frozen until the
noncollision branch of State~3 defines
\[
\Theta_m(a_1)=b^*,
\]
so clause \textup{(2a)} never acts on $a_1$ while $\Theta_m(a_1)$ is undefined;
after that moment, clause \textup{(2a)} no longer applies to $a_1$. Finally,
clause \textup{(2b)} cannot assign the value $a_1$: before $a_1$ is chosen this
is excluded by the later freshness of $a_1$, and afterwards clause
\textup{(2b)} writes only fresh values, hence never the already chosen number
$a_1$.

Thus $b^*$ is the unique point with
\[
\Lambda_m(b^*)=a_1.
\]
Since $y_v\neq b^*$, we obtain $c\neq a_1$.

Next, $c\neq a_u$ for every $2\le u\le v$. Fix such a $u$. We claim that no
point maps to $a_u$ under $\Lambda_m$.

No module instruction can assign the value $a_u$ to $\Lambda_m$. State~3
defines only $\Theta$-values, and State~4 defines no new values of $\Theta_m$
or $\Lambda_m$. In State~2, a run of some requirement
$\mathcal R_{\langle m,k'\rangle}$ can define a value of $\Lambda_m$, but it
writes only its own anchor $a'_1$. If $k'\neq k$, then
$a'_1\in\omega^{[m,k']}$ and so $a'_1\neq a_u\in\omega^{[m,k]}$; if it is a
different run of $\mathcal R_l$, freshness again ensures that its anchor is
different from $a_u$.

As for the filler, clause \textup{(1)} never applies to $a_u$, since
$a_u\in\omega^{[m,k]}$. A new preimage of $a_u$ under $\Lambda_m$ could be
created by clause \textup{(2a)} only by acting on $z=a_u$. No such earlier
action occurred, since $a_u$ was later chosen fresh above every number
previously mentioned. After $a_u$ is chosen, it is frozen until its
noncollision visit to State~3, where
\[
\Theta_m(a_u)=b^*
\]
is defined before $a_u$ is unfrozen. Hence clause \textup{(2a)} never acts on
$z=a_u$ while $\Theta_m(a_u)$ is undefined, and afterwards clause
\textup{(2a)} no longer applies to $a_u$. Finally, clause \textup{(2b)} cannot
assign the value $a_u$: before $a_u$ is chosen this is excluded by the later
freshness of $a_u$, and afterwards clause \textup{(2b)} writes only fresh
values, hence never the already chosen bait $a_u$.

Therefore no point maps to $a_u$ under $\Lambda_m$, and in particular
$c\neq a_u$ for every $2\le u\le v$. Hence $c\notin C_{l,t}$, as claimed.

\smallskip
\noindent\emph{Step 6: both terminal branches of State~4 contradict that
$\Phi_k$ reduces $W_i$ to $B_m$.}

\begin{enumerate}

\item Suppose first that
\[
c\notin A_t.
\]
Then the terminal action of State~4 enumerates the current anchor block
$C_{l,t}$ into $A$. In particular,
\[
a_v\in A,
\]
so by validity of $\Delta_j$,
\[
x_v=\Delta_j(a_v)\in W_i.
\]

On the other hand,
\[
y_v\in B_m \iff \Lambda_m(y_v)=c\in A.
\]
We now show that $c\notin A$ permanently. First, by Step~5,
$c\notin C_{l,t}$, so the terminal action of this run does not enumerate
$c$ into $A$. Furthermore, no higher-priority requirement can later perform a
terminal action, by the choice of the final run.
At stage $t$, every lower-priority requirement $\mathcal Q_q$ with
$q>2l+1$ and $q\le t$ is initialized immediately, and hence abandons any stored
witness or bait equal to $c$; thereafter, whether later in the same stage or at
later stages, any witness or bait newly chosen by a lower-priority requirement
is fresh and therefore different from $c$.
Since requirements enumerate numbers into $A$
only through their chosen witnesses or baits, no later action can enumerate
$c$ into $A$. Therefore
\[
c\notin A,
\]
so
\[
y_v\notin B_m.
\]
Thus $\Phi_k(x_v)=y_v\notin B_m$ while $x_v\in W_i$, contradicting that
$\Phi_k$ reduces $W_i$ to $B_m$.

\item Suppose instead that
\[
c\in A_t.
\]
Then already
\[
y_v\in B_m.
\]
Since $a_v$ was fresh when chosen, we had $a_v\notin A$ at that moment. Up to
stage $t$, only $\mathcal R_l$ can enumerate elements from the column
$\omega^{[m,k]}$, and the current run has performed no enumerations so far, as they
only occur at terminal actions. Hence
\[
a_v\notin A_t.
\]
In this branch the terminal action of State~4 does not enumerate $C_{l,t}$ into
$A$.
After this terminal action, $\mathcal R_l$ becomes inactive. Since this is the
final run, it is never initialized again; and no other requirement can
enumerate elements from the column $\omega^{[m,k]}$.
Therefore
\[
a_v\notin A
\]
permanently. By validity of $\Delta_j$,
\[
x_v=\Delta_j(a_v)\notin W_i.
\]
Thus $\Phi_k(x_v)=y_v\in B_m$ while $x_v\notin W_i$, again contradicting that
$\Phi_k$ reduces $W_i$ to $B_m$.
\end{enumerate}

Both branches lead to a contradiction. Therefore $\Phi_k$ is not a finite-one
reduction from $W_i$ to $B_m$.
\end{proof}

\begin{proof}[Proof of Theorem \ref{thm:main}]
By Lemma \ref{lem:finite-injury}, the construction produces a nonrecursive c.e.
set $A$.

Let $X\eqm A$. Choose $i$ with $W_i=X$, and choose $j$ such that
$\Delta_j\colon A\lem W_i$ is a total computable reduction. Put
$m=\langle i,j\rangle$.

By Lemma \ref{lem:Bm-equivalent}, the set
\[
B_m=\Lambda_m^{-1}(A)
\]
is c.e. and satisfies $B_m\eqm A$.

By Lemma \ref{lem:no-fin-one}, for every $k$, $\Phi_k$ is not a finite-one
reduction from $W_i$ to $B_m$. Since every total computable function appears as
some $\Phi_k$, it follows that $W_i\not\lfo B_m$. Therefore
$X\not\lfo B_m$.

Hence for every c.e. set $X\eqm A$ there is a c.e. set $B\eqm A$ such that
$X\not\lfo B$. This shows that the many-one degree of $A$ contains no
least finite-one degree.
\end{proof}

\section{Conclusion}

We have shown that Open Question~1 of Richter, Stephan, and Zhang has a negative answer, already within the class of computably enumerable many-one degrees. More precisely, we constructed a nonrecursive c.e.\ set $A$ such that the many-one degree of $A$ contains no least finite-one degree.

This places the general picture in a sharper perspective. On the one hand, positive answers are known for important natural subclasses, including a measure-one and comeager class of $m$-rigid sets and, in a companion paper, the c.e.\ many-one degrees containing a $D$-maximal set. On the other hand, the present construction shows that the existence of a least finite-one degree is not a universal feature of nonrecursive many-one degrees, not even in the c.e.\ setting.

It would be interesting to further clarify which structural properties of a many-one degree force the existence of a least finite-one degree, and how far the positive phenomena observed in special classes can be extended.

\section*{Acknowledgments}

This work is the result of an extended human--AI collaboration. 
Several structural ideas and technical arguments emerged from exploratory interaction with AI-based reasoning systems (Gemini Deep Think (Google DeepMind) and ChatGPT Pro (OpenAI)), which were used at different stages of the conceptual development and technical verification of this work. 
The author has fully reworked and verified all arguments and bears sole responsibility for their correctness.


\begin{thebibliography}{9}

\bibitem{Calude2002}
C.~S.~Calude,
\emph{Information and Randomness: An Algorithmic Perspective},
2nd ed., revised and extended,
Texts in Theoretical Computer Science. An EATCS Series,
Springer, Berlin, 2002.

\bibitem{ChongYu2015}
C.~T.~Chong and L.~Yu,
\emph{Recursion Theory: Computational Aspects of Definability},
De Gruyter Series in Logic and its Applications, vol.~8,
De Gruyter, Berlin/Boston, 2015.


\bibitem{Cintioli2026}
P.~Cintioli,
\emph{$m$-rigidity, finite-one, and bounded finite-one degrees inside typical many-one degrees},
arXiv:2603.02600 [math.LO], 2026.

\bibitem{DH2010}
R.~G.~Downey and D.~R.~Hirschfeldt,
\emph{Algorithmic Randomness and Complexity},
Theory and Applications of Computability,
Springer, New York, 2010.

\bibitem{LiVitanyi2008}
M.~Li and P.~M.~B.~Vit\'anyi,
\emph{An Introduction to Kolmogorov Complexity and Its Applications},
3rd ed.,
Texts in Computer Science,
Springer, New York, 2008.

\bibitem{Nies2009}
A.~Nies,
\emph{Computability and Randomness},
Oxford Logic Guides, vol.~51,
Oxford University Press, Oxford, 2009.

\bibitem{Odifreddi1981}
P.~G.~Odifreddi,
\emph{Strong reducibilities},
Bull.\ Amer.\ Math.\ Soc.\ (N.S.) \textbf{4} (1981), no.~1, 37--86.
doi:\,10.1090/S0273-0979-1981-14863-1.


\bibitem{OdifreddiCRT}
P.~G.~Odifreddi,
\emph{Classical Recursion Theory},
Studies in Logic and the Foundations of Mathematics, vol.~125,
North-Holland, Amsterdam, 1989.

\bibitem{OdifreddiCRT2}
P.~G.~Odifreddi,
\emph{Classical Recursion Theory, Volume II},
Studies in Logic and the Foundations of Mathematics, vol.~143,
Elsevier, Amsterdam, 1999.

\bibitem{Odifreddi1999}
P.~G.~Odifreddi,
\emph{Reducibilities},
in \emph{Handbook of Computability Theory},
E.~R.~Griffor (ed.),
Studies in Logic and the Foundations of Mathematics, vol.~140,
Elsevier/North-Holland, Amsterdam, 1999, pp.~89--119.
doi:\,10.1016/S0049-237X(99)80019-6.

\bibitem{Rogers1967}
H.~Rogers, Jr.,
\emph{Theory of Recursive Functions and Effective Computability},
McGraw--Hill, New York, 1967.
Reprinted by MIT Press, Cambridge, MA, 1987.

\bibitem{RSZ2026}
L.~Richter, F.~Stephan, and X.~Zhang,
\emph{Chains and antichains inside many-one degrees and variants},
preprint, 2026. Available at \url{https://linus-richter.github.io}.

\bibitem{Soare1987}
R.~I.~Soare,
\emph{Recursively Enumerable Sets and Degrees},
Perspectives in Mathematical Logic,
Springer-Verlag, Berlin, 1987.
\end{thebibliography}
\end{document}